\documentclass[reqno,12pt]{amsart}

\usepackage{amsmath, amsthm, amsfonts, amssymb}
\input{mathrsfs.sty}

\usepackage{amsmath}
\usepackage{amsfonts}
\usepackage{amssymb}
\usepackage{eufrak}
\usepackage{amscd}
\usepackage{amsthm}
\usepackage{amstext}
\usepackage[all]{xy}

\newcommand{\id}{\operatorname{id}}

   \theoremstyle{plain}
   \newtheorem{thm}{Theorem}[section]
   \newtheorem{prop}[thm]{Proposition}
   \newtheorem{lem}[thm]{Lemma}
   \newtheorem{cor}[thm]{Corollary}
   \theoremstyle{definition}
   
   \newtheorem{defn}[thm]{Definition}
   \newtheorem{example}[thm]{Example}
   \theoremstyle{remark}

\author{V. Manuilov}

\date{}

\address{Moscow State University,
Leninskie Gory 1, Moscow, 
119991, Russia}

\email{manuilov@mech.math.msu.su}

\thanks{The author acknowledges partial support by the RFBR grant.}

\title{Roe bimodules as morphisms of discrete metric spaces}

\begin{document}

\maketitle

\begin{abstract}
For two discrete metric spaces, $X$ and $Y$ we consider metrics on $X\sqcup Y$ compatible with the metrics on $X$ and $Y$. 
As morphisms from $X$ to $Y$ we consider the Roe bimodules, i.e. the norm closures of bounded finite propagation operators from $l^2(X)$ to $l^2(Y)$. We study the corresponding category $\mathcal M$, which is also a 2-category. We show that almost isometries determine morphisms in $\mathcal M$. We also consider the case $Y=X$, when there is a richer algebraic structure on the set of morphisms of $\mathcal M$: it is a partially ordered semigroup with the neutral element, with involution, and with a lot of idempotents. We also give a condition when a morphism is a $C^*$-algebra.

\end{abstract}

\section*{Introduction}








Let $X=(X,d_X)$ be a discrete countable metric space, and let $H_X=l^2(X)$ be the Hilbert space of square-summable functions on $X$, with the standard basis consisting of delta-functions $\delta_x$, $x\in X$. A bounded operator $T$ on $H_X$ with the matrix $(T_{xy})_{x,y\in X}$ has propagation less than $L$ if $d_x(x,y)\geq L$ implies that $T_{xy}=0$. The $*$-algebra of all bounded operators of finite propagation is denoted by $\mathbb C_u[X]$, and its norm completion in $\mathbb B(H_X)$ is the uniform Roe algebra $C^*_u(X)$. Our standard references are \cite{Burago} for metric space theory, and \cite{Roe} for Roe algebras. 

Let the objects of a category $\mathcal M$ be all discrete countable metric spaces.

Let $(X,d_X)$ and $(Y,d_Y)$ be two objects in $\mathcal M$. Let $Z=X\sqcup Y$, and let $D_{X,Y}$ denote the set of all metrics $d$ on $Z$ such that $d|_X=d_X$ and $d|_Y=d_Y$. Such metrics are used to define the Gromov--Hausdorff distance between $X$ and $Y$. For each $d\in D(X,Y)$, let $\mathbb M_d[X,Y]$ denote the set of all bounded finite propagation operators $T:H_X\to H_Y$, and let $M_d(X,Y)$ be its norm closure in the bimodule $\mathbb B(H_X,H_Y)$ of all bounded operators from $H_X$ to $H_Y$.

These bimodules are Hilbert $C^*$-bimodules \cite{MT} over the uniform Roe algebras $C^*_u(X)$ and $C^*_u(Y)$, and we call such bimodules {\it uniform Roe bimodules}. We show that these bimodules can be used as morphisms in the category $\mathcal M$ of discrete metric spaces.
Our aim is to study this category. It is clearly a 2-category \cite{BMZ}, though with too few 2-morphisms. We show that almost isometries 
determine morphisms in $\mathcal M$. We also consider the case $Y=X$, when there is a richer algebraic structure on the set of morphisms of $\mathcal M$: it is a partially ordered semigroup with the neutral element, with involution, and with a lot of idempotents. We also give a condition when a morphism is a $C^*$-algebra.

\section{Morphisms in $\mathcal M$}


\begin{lem} With respect to the natural action of the uniform Roe algebras by composition,
\begin{itemize}
\item
$\mathbb M_d(X,Y)$ is a $\mathbb C_u[X]$-$\mathbb C_u[Y]$-bimodule; 
\item
$M_d(X,Y)$ is a Hilbert $C^*_u(X)$-$C^*_u(Y)$-bimodule.
\end{itemize}
\end{lem}
\begin{proof}
Let $T\in\mathbb B(H_X,H_Y)$, $R\in\mathbb B(H_X)$, $S\in\mathbb B(H_Y)$ be bounded operators of finite propagation with respect to a metric $d\in D(X,Y)$. Then $R\circ T\circ S$ is obviously of finite propagation with respect to $d$.
As the inner $C^*_u(X)$-valued (resp., $C^*_u(Y)$-valued) product is given by $\langle T,S\rangle=T^*\circ S$ (resp. $\langle T,S\rangle=T\circ S^*$), it remains to check that $T^*\circ S$ and $T\circ S^*$ are of finite propagation in $\mathbb B(H_X)$ and $\mathbb B(H_Y)$ respectively, which is trivially true. 

\end{proof}

We shall call these bimodules {\it uniform Roe bimodules}. Let morphisms in the category $\mathcal M$ be all uniform Roe bimodules, i.e. all $M_d(X,Y)$, where $d\in D(X,Y)$.

Let $\mathbb K=\mathbb K(H_X,H_Y)$ denote the set of compact operators. 
\begin{lem}\label{comp}
$\mathbb K\subset M_d(X,Y)$ for any $d\in D(X,Y)$, and if $\mathbb K= M_d(X,Y)$ then there exists $L>0$ such that for any injective map $f:X\to Y$, only finite number of points $x\in X$ satisfy $d(x,f(y))\leq L$.

\end{lem}
\begin{proof}
Any finite rank operator obviously belongs to $M_d(X,Y)$ for any $d\in D(X,Y)$, hence the closure also lies in $M_d(X,Y)$. Now suppose that $M_d(X,Y)=\mathbb K$, and that for any $L>0$ there exist infinitely many points $x_n\in X$ and infinitely many points $y_n\in Y$, $n\in\mathbb N$, such that $d(x_n,y_n)\leq L$. Then $T=\sum_{n\in\mathbb N}e_{x_n,y_n}\in M_d(X,Y)$ is not compact.

\end{proof}

\begin{lem}
Let $M_d(X,Y)$ be a full right Hilbert $C^*$-module over $C^*_u(X)$. Then there exists $L>0$ and a map $f:X\to Y$ such that $d(x,f(x))\leq L$ for any $x\in X$.

\end{lem}
\begin{proof}
Recall that fullness means that the sums $\sum_{i=1}^n \langle T_i,S_i\rangle$, $S_i,T_i\in T_i,S_i\in M_d(X,Y)$, is dense in $C^*_u(X)$. The latter is unital, hence, there exists $n\in\mathbb N$, $L>0$ and $T_i,S_i\in M_d(X,Y)$ of propagation less than $L$ such that $\|1-\sum_{i=1}^nT_i^*\circ S_i\|<\frac{1}{2}$. Applying this to the diagonal matrix entries, we get $|1-\sum_{i=1}^n \sum_{y\in Y}(\overline {T}_i)_{xy}(S_i)_{yx}|<\frac{1}{2}$ for any $x\in X$. Then, for any $x\in X$, we can find $i$ such that $\sum_{y\in Y}(T_i)_{xy}S_{yx}\neq 0$. Then there is $y\in Y$ (non-uniquely defined) such that $(T_i){xy}(S_i)_{yx}\neq 0$, i.e. both $(T_i)_{xy}$ and $(S_i)_{yx}$ are non-zero. Set $f(x)=y$. As the propagation of $S_i$ and $T_i$ is less than $L$, we conclude that $d(x,f(x))\leq L$.

\end{proof}

Note that different metrics in $D(X,Y)$ can give the same uniform Roe bimodule. Recall that $X$ is uniformly discrete if $\inf_{x\neq x'}d(x,x')>0$, and $X$ is proper if each ball is compact, i.e. consists of a finite number of points. Recall that two metrics, $d$, $d'$ are coarsely equivalent (we write $d\sim_c d'$) if there exists a monotonely increasing function $f$ such that $f^{-1}(d(x,y))\leq d'(x,y)\leq f(d(x,y))$ for any $x\in X$, $y\in Y$.

\begin{lem}
Let $d,d'\in D(X,Y)$. Consider the following conditions:
\begin{itemize}
\item[(i)]
$d\sim_c d'$; 
\item[(ii)]
$M_{d'}(X,Y)=M_d(X,Y)$. 

\end{itemize}
Then (i) implies (ii). If, additionally, $X$ and $Y$ are uniformly discrete and proper then (ii) implies (i).

\end{lem}
\begin{proof}

Let (i) hold, and let $T\in M_d(X,Y)$ be of propagation less than $L$ with respect to $d$ then $d(x,y)>L$. If $d'(x,y)>L'=f(L)$ then $d(x,y)>L$, hence $T_{xy}=0$, therefore $T$ has propagation less than $L'$ with respect to $d'$, i.e. $T\in M_{d'}(X,Y)$. Interchaging $d$ and $d'$, we obtain that (ii) holds.

Now assume that $X$ and $Y$ are uniformly discrete and proper, and that $M_d(X,Y)\subset M_{d'}(X,Y)$, but there is no monotonely increasing function such that $d'(x,y)\leq f(d(x,y))$ for any $x\in X$, $y\in Y$. The latter means that there exists some $R>0$ such that the set $\{d'(x,y):d(x,y)<R\}\subset [0,\infty)$ is unbounded: indeed, if not then we may set 
$$
f(R)=\sup\{d'(x,y):d(x,y)<R\}. 
$$
Hence there are sequences $\{x_n\}_{n\in\mathbb N}\subset X$ and $\{y_n\}_{n\in\mathbb N}\subset Y$ such that $d(x_n,y_n)<R$ and $d'(x_n,y_n)>n$. We claim that by passing to a subsequence, we may assume that all $x_n$ and all $y_n$, $n\in\mathbb N$, are different. Indeed, assume that there exists $n$ such that $d(x_n,y_m)<R$ and $d'(x_n,y_m)>m$ for infinitely many $m$. Then, by the triangle inequality, $d_Y(y_n,y_m)<2R$ for infinitely many $m$ --- a contradiction.

Assuming that all $x_n$ and all $y_n$, $n\in\mathbb N$, are different, set $T=\sum_{n\in\mathbb N}e_{x_n,y_n}$, where $e_{xy}$ denotes the elementary matrix (the rank one operator taking $\delta_x\in H_X$ to $\delta_y\in H_Y$). This sum is weakly convergent, so $T$ is well defined. Obviously, $T$ has finite propagation with respect to $d$, but its distance from finite propagation operators with respect to $d'$ equals 1.  

\end{proof}

\begin{defn}
Two metrics, $d$ and $d'$, in $D(X,Y)$ are equivalent (we write $d\sim d'$) if $M_{d}(X,Y)=M_{d'}(X,Y)$.

\end{defn}

We write $\mathcal Mor(X,Y)=D(X,Y)/\sim$, and $[d]\in \mathcal Mor(X,Y)$ for the class of $d\in D(X,Y)$.

\section{Composition of morphisms} 

In order to define the composition of morphisms we need the following lemma.

\begin{lem}\label{Lemma_d}
Let $X,Y,Z\in \mathcal Ob(\mathcal M)$. Let $d_{XY}\in D(X,Y)$ be a metric on $X\sqcup Y$ and $d_{Y,Z}\in D(Y,Z)$ a metric on $Y\sqcup Z$. Then the formula
$$
d(x,z)=\inf_{y\in Y}(d_{X,Y}(x,y)+d_{Y,Z}(y,z)),\quad x\in X,z\in Z,
$$ 
defines a metric in $D(X,Z)$.

\end{lem}
\begin{proof}
Due to symmetry, we have to check the triangle inequality for $x,x'\in X$ and $z\in Z$. Fix $\varepsilon>0$ and let $y,y'\in Y$ satisfy 
$$
d_{XY}(x,y)+d_{YZ}(y,z)-d(x,z)<\varepsilon;\quad d_{XY}(x',y')+d_{YZ}(y',z)-d(x',z)<\varepsilon.
$$ 
Then 
\begin{eqnarray*}
d_X(x,x')&\leq& d_{XY}(x,y)+d_{XY}(y,x')\leq d_{XY}(x,y)+d_Y(y,y')+d_{XY}(y',x')\\
&\leq&d_{XY}(x,y)+d_{YZ}(y,z)+d_{YZ}(z,y')+d_{XY}(y',x')\leq d(x,z)+d(x',z)+2\varepsilon;\\
d(x',z)&\leq&d_{XY}(x',y)+d_{YZ}(y,z)\leq d_X(x',x)+d_{XY}(x,y)+d_{YZ}(y,z)\\
&\leq&d_X(x',x)+d(x,y)+\varepsilon.
\end{eqnarray*}
Taking $\varepsilon$ arbitrarily small, we obtain the triangle inequality, hence $d\in D(X,Z)$.

\end{proof}

We shall denote this metric by $d_{Y,Z}\circ d_{X,Y}$.

For Banach subspaces $M\subset\mathbb B(H_X,H_Y)$, $N\subset\mathbb B(H_Y,H_Z)$, let $M\hat\otimes N$ denote the norm closure, in $\mathbb B(H_X,H_Z)$, of the span of all compositions of the form $S\circ T$, where $T\in M$, $S\in N$.

\begin{lem}\label{composition}
Let $T\in\mathbb M_{d_{XY}}(X,Y)$, $S\in\mathbb M_{d_{YZ}}(Y,Z)$. Then the composition $S\circ T$ lies in $\mathbb M_d(X,Z)$, where $d=d_{Y,Z}\circ d_{X,Y}$ is determined by Lemma \ref{Lemma_d}. 

\end{lem}
\begin{proof}
Let $T$ and $S$ have propagation less than $R$. If $d(x,z)>2R$ then, for any $y\in Y$, either $d_{XY}(x,y)$ or $d_{YZ}(y,z)$ is greater than $R$, therefore, for any $y\in Y$, either $S_{yz}$ or $T_{xy}$ is zero, where $S_{yz}$ denotes the matrix entries of $S$, hence $\sum_{y\in Y}S_{zy}T_{yx}=0$, so $S\circ T$ has propagation less than $2R$.  

\end{proof}

Thus, the composition $c(T,S)=S\circ T$ defines a linear map 
$$
c:\mathbb M_{d_{XY}}(X,Y)\otimes_{\mathbb C_u[Y]}\mathbb M_{d_{YZ}}(Y,Z)\to \mathbb M_{d}(X,Z),
$$
which gives rize to the inclusion
$$
\bar c:
M_{d_{XY}}(X,Y)\hat\otimes M_{d_{YZ}}(Y,Z)\to M_{d}(X,Z).
$$

Recall that a discrete metric space $X$ has bounded geometry if for any $R>0$ the number $|B_R(x)|$ of points in a ball $B_R(x)$ of radius $R$ centered at $x\in X$ is uniformly bounded with respect to $x\in X$.

\begin{lem}
Let $X$, $Y$ and $Z$ be of bounded geometry. Then $M_{d_{XY}}(X,Y)\hat\otimes M_{d_{YZ}}(Y,Z)= M_{d}(X,Z)$.

\end{lem}
\begin{proof}
As $M_{d_{XY}}(X,Y)\hat\otimes M_{d_{YZ}}(Y,Z)$ is closed, it suffices to show that it contains all operators of finite propagation. Let $T\in\mathbb M_d(X,Z)$. Note that the condition of bounded geometry of $X$ and of $Z$ means that $T$ has the matrix with a bounded number of non-zero entries in each row and in each column, hence it can be written as a finite sum of width 1 band operators: $T=\sum_{k=1}^N T_k$, where each $T_k$ has the form $\sum_{x\in X}\lambda_xe_{x,\sigma(x)}$, where $\lambda_x\in\mathbb C$ and $\sigma:X\to Z$ is an injective map. Let us show that each $T_k$ can be written as $S\circ R$ for some $S\in\mathbb M_{d_{YZ}}(Y,Z)$, $R\in\mathbb M_{d_{XY}}(X,Y)$. For $x\in X$, $z=\sigma(x)$, let $y=f(x)\in Y$ satisfy 
$$
d_{XY}(x,y)+d_{YZ}(y,z)\leq d(x,z)+1. 
$$
Set 
$$
R=\sum_{x\in X}\lambda_xe_{x,f(x)}, \quad S=\sum_{x\in X}e_{f(x),z}. 
$$
Bounded geometry of $Y$ implies that $R$ and $S$ are bounded operators, as the number of points in the sets $f^{-1}(y)$ is uniformly bounded. It is also easy to see that both $R$ and $S$ have finite propagation. and that $T_k=S\circ R$.

\end{proof}




\begin{cor}

If $X$, $Y$, $Z$ are of bounded geometry, and if $[d_{XY}]=[d'_{XY}]$, $[d_{YZ}]=[d'_{YZ}]$ then $[d_{YZ}\circ d_{XY}]=[d'_{YZ}\circ d'_{XY}]$. 

\end{cor}


Thus we can define the map $c_{XYZ}:(M_{d_{XY}}(X,Y), M_{d_{YZ}}(Y,Z))\mapsto M_d(X,Z)$, where $d$ is determined by Lemma \ref{Lemma_d}.

Note that $D(X,Y)=D(Y,X)$, but, the same metric $d\in D(X,Y)$, considered as a morphism, gives two different morphisms --- from $X$ to $Y$ and back. Let us write $d^*$ for the morphism from $Y$ to $X$. We have $M_{d^*}(Y,X)=(M_d(X,Y))^*$.

\begin{prop}\label{inv-semi}
The metrics $d$ and $d\circ d^*\circ d$ are equivalent for any $d\in D(X,Y)$.

\end{prop}
\begin{proof}
Let $x\in X$, $y\in Y$.
On the one hand, taking $x'=x$, $y'=y$, we get
$$
(d\circ d^*\circ d)(x,y)=\inf_{x'\in X,y'\in Y}d(x,y')+d(y',x')+d(x',y)\leq 3d(x,y).
$$
On the other hand, passing to infimum in the triangle inequality, we get
\begin{eqnarray*}
(d\circ d^*\circ d)(x,y)&=&\inf_{x'\in X,y'\in Y}d(x,y')+d(y',x')+d(x',y)\geq \inf_{x'\in X}d_X(x,x')+d(x',y)\\
&\geq& d(x,y).
\end{eqnarray*}

\end{proof}

For given $X$ and $Y$, and for two uniform Roe bimodules $M_{d_1}(X,Y)$ and $M_{d_2}(X,Y)$ define a 2-morphism $\alpha:M_{d_1}(X,Y)\to M_{d_2}(X,Y)$ as a bimodule homomorphism. In fact, there are very few of them, but anyway, we get a 2-category.

\section{Almost isometries as morphisms}

\begin{defn}
A map $f:X\to Y$ is called an almost isometry if there exists $C>0$ such that 
\begin{equation}\label{eq1}
d_X(x,x')-C\leq d_Y(f(x),f(x'))\leq d_X(x,x')+C
\end{equation}
for any $x,x'\in X$.


An almost isometry $f:X\to Y$ is an almost isometric isomorphism if there exists an almost isometry $Y\to X$ and $C>0$ such that $d(g\circ f(x),x)<C$ and $d(f\circ g(y),y)<C$ for all $x\in X$, $y\in Y$. 

Let $A\subset X$. An almost isometry $f:A\to Y$ is called a partial almost isometry from $X$ to $Y$ with support $A$.

\end{defn}

Note that any isometry is an almost isometry, but quasi-isometries are usually not almost isometries. An example of an almost isometry for $X=Y=\Gamma$, where $\Gamma$ is a finitely generated group with the word-length metric, is provided by conjugation by a fixed element $g\in\Gamma$.

The following example shows that there are invertible almost isometries that are not close to any genuine isometry. Let $X\subset\mathbb Z$, $X=\cup_{n\in\mathbb Z}x_n$, where $x_n=-2n$ for $n\leq 0$ and $x_{2n}=4n$, $x_{2n-1}=4n-3$ for $n>0$. Let $Y=X$. It is clear that the only isometry of $X$ is the identity map. But the map $f:X\to X$ given by $f(x_n)=x_{-n}$, $n\in\mathbb Z$, is an almost isometry. Note that for non-discrete spaces the problem of existence of almost isometries far from genuine isometries (AI-rigidity) is more difficult, cf. \cite{Kar}.

\begin{lem}\label{def_of d_f}
Let $A\subset X$, $f:A\to Y$ be an almost isometry, and let $C$ satisfy (\ref{eq1}). For $x\in X$, $y\in Y$, set 
$$
d(x,y)=\inf_{\tilde x\in A}d_X(x,\tilde x)+C/2+d_Y(f(\tilde x),y).
$$
Then $d$ is a metric on $X\sqcup Y$, i.e. satisfies the triangle inequality.

\end{lem}
\begin{proof}
First, let $x_1,x_2\in X$, $y\in Y$. Fix $\varepsilon>0$, and let $\tilde x_1,\tilde x_2\in A$ satisfy
$$
d_X(x_1,\tilde x_1)+C/2+d_Y(f(\tilde x_1),y)<d(x_1,y)+\varepsilon;
$$
$$
d_X(x_2,\tilde x_2)+C/2+d_Y(f(\tilde x_2),y)<d(x_2,y)+\varepsilon.
$$
Then 
\begin{eqnarray*}
d_X(x_1,x_2)&\leq&d_X(x_1,\tilde x_1)+d_X(\tilde x_1,\tilde x_2)+d_X(\tilde x_2,x_2)\\
&\leq& 
d_X(x_1,\tilde x_1)+(d_Y(f(\tilde x_1),f(\tilde x_2))+C)+d_X(\tilde x_2,x_2)\\
&\leq&d_X(x_1,\tilde x_1)+C/2+d_Y(f(\tilde x_1),y)+d_X(x_2,\tilde x_2)+C/2+d_Y(f(\tilde x_2),y)\\
&\leq& d(x_1,y)+d(x_2,y)+2\varepsilon.
\end{eqnarray*}
(here the first inequality is the triangle inequality for $d_X$, the second inequality follows from almost isometricity of $f$, the third inequality is the triangle inequality for $d_Y$, and the last inequality follows from our choice of the points $\tilde x_1$ and $\tilde x_2$). As $\varepsilon$ is arbitrary, this implies that $d_X(x_1,x_2)\leq d(x_1,y)+d(y,x_2)$.

\begin{eqnarray*}
d(x_1,y)&=&\inf_{\tilde x\in A}d_X(x_1,\tilde x)+C/2+d_Y(f(\tilde x),y)\\
&\leq&d_X(x_1,\tilde x_2)+C/2+d_Y(f(\tilde x_2),y)\\
&\leq& d_X(x_1,x_2)+d_X(x_2,\tilde x_2)+C/2+d_Y(f(\tilde x_2),y)\\
&\leq&d_X(x_1,x_2)+d(x_2,y)+\varepsilon
\end{eqnarray*}
(the infimum here is estimated by setting $\tilde x=\tilde x_2$, which gives the first inequality, the second inequality is the triangle inequality for $d_X$, and the last inequality follows from our choice of the point $\tilde x_2$),
hence $d(x_1,y)\leq d_X(x_1,x_2)+d(x_2,y)$.

Now let $x\in X$, $y_1,y_2\in Y$. Let $\tilde x_1,\tilde x_2\in A$ satisfy
$$
d_X(x,\tilde x_1)+C/2+d_Y(f(\tilde x_1),y_1)<d(x,y_1)+\varepsilon;
$$
$$
d_X(x,\tilde x_2)+C/2+d_Y(f(\tilde x_2),y_2)<d(x,y_2)+\varepsilon.
$$
Then
\begin{eqnarray*}
d_Y(y_1,y_2)&\leq&d_Y(y_1,f(\tilde x_1)+d_Y(f(\tilde x_1),\tilde x_2)+d_Y(\tilde x_2,y_2)\\
&\leq&d_Y(y_1,f(\tilde x_1))+(d_X(\tilde x_1,\tilde x_2)+C)+d_Y(\tilde x_2,y_2)\\
&\leq& d_Y(y_1,f(\tilde x_1))+C/2+d_X(\tilde x_1,x)+d_X(x,\tilde x_2)+C/2+d_Y(\tilde x_2,y_2)\\
&\leq& d(y_1,x)+d(y_2,x)+2\varepsilon
\end{eqnarray*}
(the first inequality is the triangle inequality for $d_Y$, the second inequality follows from almost isometricity of $f$, the third inequality is the triangle inequality for $d_X$, and the last inequality follows from our choice of the points $\tilde x_1$ and $\tilde x_2$). This implies that $d_Y(y_1,y_2)\leq d(y_1,x)+d(x,y_2)$.

Finally,
\begin{eqnarray*}
d(x,y_1)&=&\inf_{\tilde x\in A}d_X(x,\tilde x)+C/2+d_Y(f(\tilde x),y_1)\\
&\leq&d_X(x,\tilde x_2)+C/2+d_Y(f(\tilde x_2),y_1)\\
&\leq&d_X(x,\tilde x_2)+C/2+d_Y(f(\tilde x_2),y_2)+d_Y(y_2,y_1)\\
&\leq&d(x,y_2)+\varepsilon+d_Y(y_2,y_1)
\end{eqnarray*}
(the infimum is majorized by evaluation at $\tilde x=\tilde x_2$, hence the first inequality, the second inequality is the triangle inequality for $d_Y$, and the last inequality follows from our choice of the point $\tilde x_2$), hence $d(x,y_1)\leq d(x,y_2)+d(y_2,y_1)$ holds.

\end{proof} 

Given a partial almost isometry $f:X\to Y$, we denote the metric on $X\sqcup Y$ defined above by $d_f$. Note that $d_f$ depends on the constant $C$, but the equivalence class $[d_f]$ doesn't. 

\begin{lem}\label{comp}
Let $f:X\to Y$, $g:Y\to Z$ be almost isometries. There exists $C>0$ such that $d_g\circ d_f(x,z)\leq d_{g\circ f}(x,z)\leq d_g\circ d_f(x,z)+C$ for any $x\in X$, $z\in Z$.

\end{lem}
\begin{proof}
Let $C$ satisfy 
$$
|d_Y(f(x),f(x'))-d_X(x,x')|\leq C;\quad |d_Z(g(y),g(y'))-d_Y(y,y')|\leq C
$$
for any $x,x'\in X$, $y,y'\in Y$.
Then $|d_Z(g\circ f(x),f(x'))-d_X(x,x')|\leq 2C$, hence $g\circ f$ is also an almost isometry.

Recall that 
\begin{equation}\label{eq2}
d_{g}\circ d_f(x,z)=\inf_{y\in Y}(\inf_{\tilde x\in X}(d_X(x,\tilde x)+d_Y(f(\tilde x),y))+\inf_{\tilde y\in Y}(d_Y(y,\tilde y)+d_Z(g(\tilde y),z))+C;
\end{equation}
$$
d_{g\circ f}(x,z)=\inf_{\tilde x\in X}(d_X(x,\tilde x)+d_Z(g\circ f(\tilde x),z))+C.
$$

Fix $\varepsilon>0$, and let $x'\in X$ satisfy $d_{g\circ f}(x,z)\geq d_X(x,x')+d_Z(g\circ f(x'),z)+C-\varepsilon$.

Taking $\tilde x=x'$, $y=f(x')=\tilde y$ in (\ref{eq2}), we see that
\begin{eqnarray*}
d_{g}\circ d_f(x,z)&\leq& d_X(x,x')+d_Y(f(x'),f(x'))+d_Y(y,y)+d_Z(g(f(x')),z)+C\\
&=&d_X(x,x')+d_Z(g(f(x')),z)+C\leq d_{g\circ f}(x,z)+\varepsilon,  
\end{eqnarray*}
hence $d_{g}\circ d_f(x,z)\leq d_{g\circ f}(x,z)$ for any $x\in X$, $z\in Z$.

On the other hand,
\begin{eqnarray*}
d_{g\circ f}(x,z)&\leq&d_X(x,\tilde x)+d_Z(g\circ f(\tilde x),z)+C\\
&\leq&d_X(x,\tilde x)+d_Z(g\circ f(\tilde x),g(y))+d_Z(g(y),g(\tilde y))+d_Z(g(\tilde y),z)+C\\
&\leq&d_X(x,\tilde x)+d_Y(f(\tilde x),y)+C+d_Y(y,\tilde y)+C+d_Z(g(\tilde y),z)+C
\end{eqnarray*}
for any $\tilde x\in X$, $y,\tilde y\in Y$, hence taking infimum for the right-hand side, we obtain that
$d_{g\circ f}(x,z)\leq d_g\circ d_f(x,z)+C$.

\end{proof}

Thus, by Lemma \ref{comp}, $[d_{g\circ f}]=[d_g\circ d_f]$. 
As a corollary, we get the following statement.
\begin{thm}
Let $\mathcal M_{AI}$ denote the category of discrete metric spaces of bounded geometry with almost isometries as morphisms. Then the mapping $f\mapsto M_{d_f}$ is a functor from $\mathcal M_{AI}$ to $\mathcal M$.

\end{thm}

Note that in $\mathcal M$ we may compose also partial almost isometries, even if the range of the first one and the domain of the second one are disjoint.

\section{Case $Y=X$}

In this section we are interested in the case when $Y=X$ (with the same metric). In this case there is a special metric $d^0\in D(X,X)$ defined as follows. Let $Z=X\times\{0,1\}$. To distinguish between the two copies of $X$ in $Z$ we shall write $x$ for $(x,0)$ and $x'$ for $(x,1)$. For $x_1,x_2\in X$, set $d^0(x_1,x'_2)=d_X(x_1,x_2)+1$. Note that $M_{d^0}(X,X)=C^*_u(X)$.

\begin{lem}\label{XX}
Let $d\in D(X,X)$, and let $Id\in M_d(X,X)$. Then the metrics $d$ and $d^0$ are almost isometric.

\end{lem}
\begin{proof}
As finite propagation operators are dense in $M_d(X,X)$, there exists $L>0$ and operator $T\in M_d(X,X)$ of propagation $\leq L$ such that $\|Id-T\|<\frac{1}{2}$. Then $|T_{x,x'}-1|<\frac{1}{2}$ for any $x\in X$. The latter implies that $d(x,x')\leq L$. Using the triangle inequality, we get
$$
d(x_1,x'_2)= d_X(x_1,x_2)+d(x_2,x'_2)\leq d_X(x_1,x_2)+L=d^0(x_1,x'_2)+L-1;
$$ 
$$
d^0(x_1,x'_2)=d_X(x_1,x_2)+1\leq d(x_1,x'_2)+d(x_2,x'_2)+1\leq d(x_1,x'_2)+L+1.
$$

\end{proof}

\begin{thm}
Let $d\in D(X,Y)$ and $d'\in D(Y,X)$ satisfy $M_{d'\circ d}(X,X)=M_{d^0}(X,X)$.
Then there exists an almost isometry $f:X\to Y$ such that $d$ and $d_f$ are almost isometric.

\end{thm}
\begin{proof}
By Lemma \ref{XX}, $d'\circ d$ and $d^0$ are almost isometric, and there exists $L>0$ such that for any $x\in X$, $\inf_{y\in Y}d(x,y)+d'(y,x)<L$. Therefore, for any $x\in X$ there exists $y\in Y$ such that $d(x,y)<L$. Set $f(x)=y$.   

It remains to check that the two metrics are almost isometric. Taking $\tilde x=x$, we get
\begin{eqnarray*}
d_f(x,y)&=&\inf_{\tilde x\in X}d_X(x,\tilde x)+C/2+d_Y(f(\tilde x),y)\leq
C/2+d_Y(f(x),y)\\
&\leq& C/2+d(f(x),x)+d(x,y)\leq C/2+L+d(x,y).
\end{eqnarray*}

On the other hand,
\begin{eqnarray*}
d(x,y)&\leq& d_X(x,\tilde x)+d(\tilde x,f(\tilde x))+d_Y(f(\tilde x),y)\leq d_X(x,\tilde x)+L+d_Y(f(\tilde x),y)\\
&=&d_X(x,\tilde x)+C/2+d_Y(f(\tilde x),y)+(L-C/2)
\end{eqnarray*}
holds for any $\tilde x\in X$, hence we may pass to the infimum and obtain that $d(x,y)\leq d_f(x,y)+L-C/2$.

\end{proof}

\begin{cor}
If $d\in D(X,Y)$, $d'\in D(Y,X)$, $M_{d\circ d'}=M_{d^0}(Y,Y)$, $M_{d'\circ d}=M_{d^0}(X,X)$. Then $X$ and $Y$ are almost isometric.

\end{cor}

For a metric space $(X,d_X)$ consider the set $\mathcal Mor(X,X)$. It is a semigroup with the neutral element $[d^0]$ (or, equivalently, $M_{d^0}(X,X)=C^*_u(X)$), and with the involution $d\mapsto d^*$, where $d^*(x_1,x_2')=d(x_2,x_1')$. An element $[d]\in\mathcal Mor(X,X)$ is selfadjoint if $M_{d^*}(X,X)=M_d(X,X)$, and is an idempotent if $M_{d\circ d}(X,X)=M_d(X,X)$.
Note that, by Proposition \ref{inv-semi}, the metric $d^*\circ d\in D(X,X)$ is a selfadjoint idempotent for any $d\in D(X,Y)$.

\begin{lem}
Let $d$ be a selfadjoint idempotent metric. Then $M_d(X,X)$ is a $C^*$-algebra. In particular, $M_{d^*\circ d}(X,X)$ and $M_{d\circ d^*}(Y,Y)$ are $C^*$-algebras for any $d\in D(X,Y)$.

\end{lem}
\begin{proof}
$M_d(X,X)$ is closed under taking products and adjoints.

\end{proof}

Note that there are typically a lot of idempotent metrics (see Example \ref{idem}) below. This means that there is no cancellation in the semigroup $\mathcal Mor(X,X)$. Indeed, if $d\circ d\approx d$ then cancellation implies $d\approx d^0$.

Note also that Proposition \ref{inv-semi} implies that for any $d\in D(X,X)$ there exists $d'\in D(X,X)$ such that $[d\circ d'\circ d]=[d]$, so $\mathcal Mor(X,X)$ is a regular semigroup.

\begin{example}\label{idem}
Let $X=\mathbb Z$ with the standard metric $d(n,m)=|n-m|$, $n,m\in\mathbb Z$. Let $f:\mathbb N\to\mathbb N$ denote the identity map. We can consider $f$ as a partial almost isometry from $\mathbb Z$ to itself. Let $d_f$ be the metric on the union of two copies of $\mathbb Z$ determined by $f$. Then it is easy to see that $d_f^*=d_f$ and $d_f\circ d_f=d_{f\circ f}=d_f$, while $d_f$ is not coarsely equivalent to $d^0$.
The $C^*$-algebra $M_{d_f}(X,X)$ is isomorphic to the sum of $C^*_u(\mathbb N)$ and the algebra of compact operators.

\end{example}


\section{Partial order on $\mathcal Mor(X,Y)$}

Given $d,d'\in D(X,Y)$ we say that $d\preceq d'$ if $M_d(X,Y)\subset M_{d'}(X,Y)$. For example, if $d(x,y)\geq d'(x,y)$ for any $x\in X$, $y\in Y$ then $d\preceq d'$. If $d_1,d_2\in D(X,Y)$ then $d$ defined by $d(x,y)=\max(d_1(x,y),d_2(x,y))$ satisfies $d\preceq d_1$, $d\preceq d_2$, so the partially ordered set $D(X,Y)$ is downwards directed. The following example shows that it is not upwards directed.

\begin{example}
Let $X=Y=\mathbb Z$, $f_1=\id_{\mathbb Z}$, $f_2(n)=-n$, $n\in\mathbb Z$. If there would exist a metric $d\in D(X,Y)$ such that $d_{f_1}\preceq d$ and $d_{f_2}\preceq d$ then $M_d(X,Y)$ would contain $T=\sum_{n\in\mathbb Z}e_{n,n'}+e_{n,-n'}$ (we denote by $x'$ the point $x$ in the second copy of $X$). Then there exists $L>0$ such that $d(n,n')\leq L$ and $d(n,-n')\leq L$ for any $n\in\mathbb Z$. If $n>L$ then this contradicts the triangle inequality for the triangle $n$, $n'$, $-n'$.

\end{example}

The partial order is compatible with the algebraic operations:
\begin{lem}
If $d_1,d_2\in D(,Y)$, $d'_1,d'_2\in D(Y,Z)$, $d_1\preceq d_2$, $d'_1\preceq d'_2$ then 
\begin{enumerate}
\item
$d'_1\circ d_1\preceq d'_2\circ d_2$;
\item
$d^*_1\preceq d^*_2$.
\end{enumerate}
\end{lem}
\begin{proof}
Obvious.

\end{proof}

It would be interesting to have a description of maximal elements in $\mathcal Mor(X,Y)$. In general this is difficult, but here are two cases, when this is easy. Let $X$ and $Y$ be almost isometric, with an almost isometry $f:X\to Y$. 

\begin{prop}
$d_f$ is a maximal element in $\mathcal Mor(X,Y)$ for any almost isometry $f$.

\end{prop}
\begin{proof}
Suppose that there exists $d\in D(X,Y)$ such that $d_f\preceq d$. Then thre is a monotonely increasing function $h$ such that
$d(x,y)\leq h(d_f(x,y))$ for any $x\in X$, $y\in Y$. Let $C$ be the constant from the definition of $d_f$ in Lemma \ref{def_of d_f}. As $d_f(x,f(x))=C/2$, we have $d(x,f(x))\leq h(C/2)$ for any $x\in X$. Then
\begin{eqnarray*}
d(x,y)&\geq& d_Y(f(x),y)-d(x,f(x))\geq d_Y(f(x),y)-h(C/2)\\
&=&d_Y(f(x),y)+C/2-(C/2+h(C/2))\\
&\geq& d_f(x,y)-(C/2+h(C/2)),
\end{eqnarray*}
which implies that $d\preceq d_f$, hence $d\sim d_f$.

\end{proof}

\begin{prop}
If $d^*\circ d$ is maximal for some $d\in D(X,Y)$ then $d^*\circ d\sim d^0$. Moreover, in this case there exists an almost isometry (not necessarily invertible) $f:X\to Y$ such that $d_f\sim d$.

\end{prop}
\begin{proof}
Let $x,y,z\in X$. We write $x'$ (resp. $x''$) for $x$ in the second (resp. the third) copy of $X$. Then
\begin{eqnarray*}
d^*\circ d(x,z'')&=&\inf_{y\in X}d(x,y')+d^*(y',z'')=\inf_{y\in X}d(x,y')+d^*(y,z')\\
&=&\inf_{y\in X}d(x,y')+d(y',z)\geq d_X(x,z)=d^0(x,z')-1.
\end{eqnarray*}
Thus, $d^*\circ d\preceq d^0$. Maximality implies that $d^*\circ d\sim d^0$.

In particular, there is a monotonely increasing map $h$ such that $d^*\circ d(x,z'')\leq h(d^0(x,z'))$. Then, in the case when $z=x$, we have $d^*\circ d(x,x'')\leq L$ for some $L>0$. Then, for any $x\in X$, there exists $y\in Y$ such that $d(x,y)\leq L$. Set $f(x)=y$. The triangle inequality for points $x,\tilde x\in X$, $f(x),f(\tilde x)\in Y$ gives
$$
d(x,\tilde x)-2L\leq d(f(x),f(\tilde x))\leq d(x,\tilde x)+2L,
$$  
hence $f$ is an almost isometry.

By the triangle inequality we have $d_f(x,y)\geq d(x,y)$ for any $x\in X$, $y\in Y$. Taking $\tilde x=x$, we have
$$
d_f(x,y)=\inf_{\tilde x\in X}d_X(x,\tilde x)+d(\tilde x,f(\tilde x))+d_Y(f(\tilde x),y)\leq L+d(f(x),y),
$$
and by the triangle inequality we have 
$$
d_Y(f(x),y)\leq d(x,f(x))+d(x,y)\leq d(x,y)+L,
$$
therefore, $d_f(x,y)\leq d(x,y)+2L$, hence $d_f\sim d$.

\end{proof}

\section{Concluding remarks}

1. Parallel to the uniform Roe bimodules over uniform Roe algebras, we can consider Roe bimodules over Roe algebras.

2. Some our results are true only in the case of metric spaces of bounded geometry. More general case can be included into our considerations using the approach from \cite{MZ}.

3. More general coarse maps than almost isometries can be incorporated as morphisms if we widen the latter. Namely, given $(X,d_X)$ and $(Y,d_Y)$, we should consider coarse equivalence class of metrics $[d]$ on $X\sqcup Y$ such that for any $d'_X\sim_c d_X$ there exists $d_1\in[d]$ such that $d_1|_X=d'_X$ and $d_1|_Y\sim_c d_Y$, and for any $d'_Y\sim_c d_Y$ there exists $d_2\in[d]$ such that $d_2|_X\sim_c d_X$ and $d_2|_Y=d'_Y$.  

4. As uniform Roe algebras and uniform Roe bimodules contain all compact operators, we can pass to the quotients and get the quotient $C^*_u(X)/\mathbb K-C^*_u(Y)/\mathbb K$-bimodule $M_d(X,Y)/\mathbb K$ for $d\in D(X,Y)$. Using these quotient Roe bimodules as morphisms, we can get more 2-morphisms. For example, if we take $X=Y=\{n^2:n\in\mathbb N\}$ then $M_{d^0}(X,X)=l^\infty(\mathbb N)+\mathbb K$, where $l^\infty(\mathbb N)$ stands for diagonal operators, has only scalar bimodule homomorphisms, while $M_{d^0}(X,X)/\mathbb K\cong l^\infty(\mathbb N)/c_0$ has more homomorphisms.

\end{document}